\documentclass[a4paper,reqno]{amsart}

\usepackage[utf8]{inputenc}
\usepackage{amsmath}
\usepackage{amsfonts}
\usepackage{palatino}
\usepackage{a4wide}
\usepackage{amssymb}
\usepackage{amsthm}
\usepackage{microtype}
\usepackage{mathtools}

\theoremstyle{plain}

\newtheorem{teo}{Theorem}[section]
\newtheorem{lemma}[teo]{Lemma}
\newtheorem{prop}[teo]{Proposition}
\newtheorem{cor}[teo]{Corollary}
\newtheorem{ackn}{Acknowledgments\!}

\theoremstyle{definition}
\newtheorem{dfnz}[teo]{Definition}

\theoremstyle{remark}

\newtheorem{rem}[teo]{Remark}

\DeclareMathOperator{\Hess}{Hess}

\theoremstyle{definition}

\theoremstyle{remark}

\numberwithin{equation}{section}

\def\RR{{\mathbb R}}
\def\H{{{\mathcal H}}}
\def\R{{\mathbb R}}

\def\Ric{{\mathrm {Ric}}}

\def\trace{\operatornamewithlimits{trace}\nolimits}

\def\cutp{{\mathrm{Cut}}_p}

\def\ress  {\begin{picture}(9,8)
\put (1,0){\line(0,1){6}}
\put (1,0){\line(1,0){5}}
\end{picture}}
\def\res{\,\ress}

\def\Derpar#1 { \frac{\partial~ } {\partial {#1} }} 
\def\derpar#1 #2 { \frac{\partial{#2}} {\partial {#1} }} 

\def\lra{\to}
\def\rl{\R}
\def\real{\R}
\def\chart{\psi}
\def\composed{\circ}
\def\supdiff{{\partial^+}}
\DeclareMathOperator{\Sing}{Sing}
\DeclareMathOperator{\Conj}{Conj}
\def\testfun{\varphi}
\def\apT{{\mathrm{ap}}T}

\title[On the Distributional Hessian of the Distance Function]{On the
  Distributional Hessian of the Distance Function}

\author[Carlo Mantegazza]{Carlo Mantegazza}
\address[Carlo Mantegazza]{Scuola Normale Superiore di Pisa, Piazza dei Cavalieri 7, Pisa, Italy, 56126}
\email[C. Mantegazza]{c.mantegazza@sns.it}

\author[Giovanni Mascellani]{Giovanni Mascellani}
\address[Giovanni Mascellani]{Scuola Normale Superiore di Pisa, Piazza dei Cavalieri 7, Pisa, Italy, 56126}
\email[G. Mascellani]{giovanni.mascellani@sns.it}

\author[Gennady Uraltsev]{Gennady Uraltsev}
\address[Gennady Uraltsev]{Scuola Normale Superiore di Pisa, Piazza dei Cavalieri 7, Pisa, Italy, 56126}
\email[G. Uraltsev]{gennady.uraltsev@sns.it}

\date{\today}

\begin{document}

\begin{abstract} We describe the precise structure of the
  distributional Hessian of the distance function from a point of a
  Riemannian manifold. In doing this we also discuss some geometrical
  properties of the cutlocus of a point and we compare some different
  weak notions of Hessian and Laplacian.
\end{abstract}

\maketitle

\section{Introduction}

Let $(M,g)$ be an $n$--dimensional, smooth, complete Riemannian
manifold, for any point $p\in M$ we define $d_p:M\to\RR$ to be the
distance function from $p$. 

Such distance functions and their relatives, the Busemann functions,
enters in several arguments of differential geometry. It is easy to
see that, apart from the obvious singularity at the point $p$, with
some few exceptions such distance function is not smooth in
$M\setminus\{p\}$ (for instance, when the manifold $M$ is compact), it
is anyway 1--Lipschitz and differentiable with a unit gradient almost
everywhere (by Rademacher's theorem).

In this note we are concerned with the precise description of the
distributional Hessian of $d_p$, having in mind the following {\em
  Laplacian and Hessian comparison theorems} (see~\cite{petersen1},
for instance).

\begin{teo}\label{ggg} If $(M,g)$ satisfies $\Ric\geq (n-1)K$ then, 
considering polar coordinates around the points $p\in M$ and $P$ in 
the simply connected, $n$--dimensional
space $S^K$ of constant curvature $K\in\R$, we have
$$
\Delta d_p(r)\leq \Delta^K d^K_{P}(r)\,.
$$
If the sectional curvature of $(M,g)$ is greater or equal to $K$, then
$$
{\mathrm{Hess}} \,d_p(r)\leq {\mathrm{Hess}}^K d^K_{P}(r)\,.
$$
Here $\Delta^K d^K_{P}(r)$ and ${\mathrm{Hess}}^K d^K_{P}(r)$
denote respectively the Laplacian and the Hessian 
of the distance function $d^K_P(\,\cdot\,)=d^K(P,\,\cdot\,)$ in $S^K$,
at distance $r$ from $P$.
\end{teo}
It is often stated by several authors that these inequalities actually
hold on the whole manifold $(M,g)$, in some weak sense, that is, in
sense of distributions, of viscosity, of barriers. Such conclusion can
simplify and sometimes is actually necessary in global arguments
involving this comparison theorem, more in general, one often would
like to use (weak or strong) maximum principle for the Laplacian in
situations where the functions involved are not smooth, for instance,
in the proof of the ``splitting'' theorem (first proved by Cheeger and
Gromoll~\cite{chegro1}) by Eschenburg and Heintze~\cite{escheintz},
but also of Topogonov theorem and the ``soul'' theorem (Cheeger,
Gromoll and Meyer~\cite{chegro2,gromey}).\\
To be precise, we give the respective definitions of these
notions.

\begin{dfnz}\label{weakdef}
Let $A$ be a smooth, symmetric $(0,2)$--tensor field on a Riemannian manifold $(M,g)$.
\begin{itemize} 
\item We say that the function $f:M\to\R$ satisfies ${\mathrm{Hess}}\,
  f\leq A$ in {\em distributional sense} if for every smooth vector
  field $V$ with compact support there holds 
$\int_M  f\,\nabla^2_{ji}(V^iV^j)\,d{\mathrm{Vol}}\leq \int_M A_{ij}V^iV^j\,d{\mathrm{Vol}}$.
\item For a continuous function $f:M\to\R$, we say that
  ${\mathrm{Hess}}\, f\leq A$ at the point $p\in M$ in {\em barrier
    sense} if for every $\varepsilon>0$ there exists a 
neighborhood $U_\varepsilon$ of the point $p$ and a $C^2$--function
$h_\varepsilon:U_\varepsilon\to\R$ such that $h_\varepsilon(p)=f(p)$,
$h_\varepsilon\geq f$ in $U_\varepsilon$ and ${\mathrm{Hess}}\,
h_\varepsilon(p)\leq A(p)+\varepsilon g(p)$ as $(0,2)$--tensor fields (such a function
$h_\varepsilon$ is called an {\em upper barrier}).
\item For a continuous function $f:M\to\R$, we say that
  ${\mathrm{Hess}}\, f\leq A$ at the point $p\in M$ in {\em viscosity
    sense} if for every $C^2$--function $h$ from a neighborhood $U$ of
  the point $p$ such that $h(p)=f(p)$ and $h\leq f$ in $U$, we have
  ${\mathrm{Hess}}\, h(p)\leq A(p)$.
\end{itemize}
The weak notions of the inequality $\Delta f\leq\alpha$, for some
smooth function $\alpha:M\to\R$, are defined analogously.
\begin{itemize} 
\item We say that the function $f:M\to\R$ satisfies $\Delta f\leq\alpha$ in {\em distributional sense} if for every smooth, nonnegative 
function $\varphi:M\to\R$ with compact support there holds 
$\int_M  f\,\Delta\varphi\,d{\mathrm{Vol}}\leq \int_M \alpha\varphi\,d{\mathrm{Vol}}$.
\item For a continuous function $f:M\to\R$, we say that
  $\Delta f\leq\alpha$ at the point $p\in M$ in {\em barrier
    sense} if for every $\varepsilon>0$ there exists a 
neighborhood $U_\varepsilon$ of the point $p$ and a $C^2$--function
$h_\varepsilon:U_\varepsilon\to\R$ such that $h_\varepsilon(p)=f(p)$,
$h_\varepsilon\geq f$ in $U_\varepsilon$ and $\Delta h_\varepsilon(p)\leq \alpha(p)+\varepsilon$.
\item For a continuous function $f:M\to\R$, we say that
  $\Delta f\leq\alpha$ at the point $p\in M$ in {\em viscosity
    sense} if for every $C^2$--function $h$ from a neighborhood $U$ of
  the point $p$ such that $h(p)=f(p)$ and $h\leq f$ in $U$, we have
  $\Delta h(p)\leq\alpha(p)$.
\end{itemize}
\end{dfnz}

\medskip

\emph{In this definition and in the following of this paper we will
  use the Einstein summation convention on repeated indices. In
  particular, with the notation $\nabla^2_{ij}(V^i V^j)$ we mean
  $\nabla^2_{ij}{(V \otimes V)}^{ij}$; i.e., the function obtained by
  contracting twice the second covariant derivative of the tensor
  product $V \otimes V$.}

\medskip

The notion of inequality ``in barrier sense'' was defined by Calabi~\cite{calabi} (for the Laplacian)
back in 1958 (he used the terminology ``weak sense'' rather than
``barrier sense'') who also proved the relative global ``weak'' Laplacian
comparison theorem (see also the book of Petersen~\cite[Section~9.3]{petersen1}).\\
The notion of viscosity solution (which is connected to the definition of 
inequality ``in viscosity sense'', see  Appendix~\ref{appA}) 
was introduced by Crandall and Lions~\cite[Definition~3.2]{crisli3} for partial differential {\em
  equations}, the above definition for the Hessian is a generalization to a
very special {\em system} of PDEs.\\
The distributional notion is useful when integrations (by parts) are
involved, the other two concepts when the arguments are based on maximum
principle.\\
It is easy to see, by looking at the definitions, that ``barrier sense''
implies ``viscosity sense'', moreover, by the work~\cite{ishii2}, if
$f:M\to\R$ satisfies $\Delta f\leq \alpha$ in viscosity sense it also
satisfies $\Delta f\leq \alpha$ as distributions and viceversa. 
In the Appendix~\ref{appA} we will discuss in detail the relations between these
definitions.

In the next section we will describe the distributional
structure of the Hessian (and hence of the Laplacian) of $d_p$ which
will imply the mentioned validity of the above inequalities on the
whole manifold.

It is a standard fact that the function $d_p$ is smooth in the set
$M\setminus(\{p\}\cup\cutp)$, where $\cutp$ is the {\em cutlocus} of
the point $p$, which we are now going to define along with stating its basic
properties (we keep the books~\cite{gahula} and~\cite{sakai} as
general references). It is anyway well known that $\cutp$ is a closed
set of zero (canonical) measure. Hence, in the open set
$M\setminus(\{p\}\cup\cutp)$ the Hessian and Laplacian of $d_p$ are
the usual ones (even seen as distributions or using other weak
definitions) and all the analysis is concerned to what happens on
$\cutp$ (the situation at the point $p$ is straightforward as $d_p$ is
easily seen to behave as the function $\Vert x\Vert$ at the origin of
$\R^n$).

We let $U_p=\left\{v\in T_pM\,\vert \, g_p(v,v)=1 \right\}$ to be the
set of unit tangent vectors to $M$ at $p$. Given $v\in U_p$ we
consider the geodesic $\gamma_v(t)=\exp_p(tv)$ and we let
$\sigma_v\in\R^+$ (possibly equal to $+\infty$) to be the maximal time
such that $\gamma_v([0,\sigma_v])$ is minimal between any pair of its
points. It is so defined a map $\sigma:U_p\to\R^+\cup\{+\infty\}$ and
the point $\gamma_v(\sigma_v)$ (when $\sigma_v<+\infty$) is called the
{\em cutpoint} of the geodesic $\gamma_v$.

\begin{dfnz}
The set of all cutpoints $\gamma_v(\sigma_v)$ for $v\in U_p$ with 
$\sigma_v<+\infty$ is called the {\em cutlocus} of the point $p\in M$.
\end{dfnz}

The reasons why a geodesic ceases to be minimal are explained in the
following proposition.

\begin{prop}
If for a geodesic $\gamma_v(t)$ from the point $p\in M$ 
we have $\sigma_v<+\infty$, at least one of the 
following two (mutually non exclusive) conditions is satisfied:
\begin{enumerate}
\item at the cutpoint $q=\gamma_v(\sigma_vv)$ there arrives at least another minimal 
  geodesic from $p$,
\item the differential $d\exp_p$ is not invertible at the point $\sigma_vv\in T_pM$.
\end{enumerate}
Conversely, if at least one of these conditions is satisfied the 
geodesic $\gamma_v(t)$ cannot be minimal on an interval larger that
$[0,\sigma_v]$.
\end{prop}

It is well known that the subset of points $q\in\cutp$ where more than
a minimal geodesic from $p$ arrive coincides with $\Sing$, which is
the singular set of the distance function $d_p$ in
$M\setminus\{p\}$. We also define the set $\Conj$ of the points
$q=\gamma_v(\sigma_v)\in\cutp$ with $d\exp_p$ not invertible at
$\sigma_vv\in T_pM$, we call $\Conj$ the {\em locus of optimal
  conjugate points} (see~\cite{gahula,sakai}).

\medskip

\begin{ackn} 
We thank Giovanni Alberti, Luigi Ambrosio, Davide Lombardo and 
Federico Poloni for several valuable
suggestions. The first author is partially supported by the Italian project
FIRB--IDEAS ``Analysis and Beyond''.
\end{ackn}

\section{The Structure of the Distributional Hessian of the Distance Function}

The following properties of the function $d_p$ and of the cutlocus of $p\in M$ are proved in the paper~\cite{mant4}, Section~3 (see also the wonderful work~\cite{nirenli} for other fine properties, notably the local Lipschitzianity of the function $\sigma:U_p\to\R^+\cup\{+\infty\}$, in Theorem~1.1). 

\smallskip

Given an open set $\Omega\subset\rl^n$, we say that a continuous function $u:\Omega\to\R$ is {\em locally semiconcave} if, for any open
convex set $K\subset\Omega$ with compact closure in $\Omega$, the function $u\vert_K$ is the sum of a $C^2$ function with a concave function.\\
A continuous function $u:M \to\rl$ is called {\em locally semiconcave} if, for any local chart
$\chart:\real^n\lra U\subset M$, the function $u\composed\chart$
is locally semiconcave in $\R^n$ according to the above definition.

\begin{prop}[{Proposition~3.4 in~\cite{mant4}}]\label{carlosemicref}
The function $d_p$ is locally semiconcave in $M\setminus\{p\}$.
\end{prop}

This fact, which follows by recognizing $d_p$ as a viscosity solution of the {\em eikonal equation} $\vert\nabla u\vert=1$ (see~\cite{mant4}), has some relevant consequences, we need some definitions for the precise statements.

\smallskip

Given a continuous function $u:\Omega \to \rl$ and a point $q\in M$, the
{\em superdifferential} of $u$ at $q$ is the subset of $T_q^*M$ defined by
$$
\supdiff u(q)=\left\{  d\testfun(q) \,|\,   \testfun\in C^1(M),
  \testfun(q)-u(q)=\min_M \testfun-u \right\}\,.
$$
For any locally Lipschitz function $u$, the set $\supdiff u(q)$ is a 
compact convex set, almost everywhere coinciding with
the differential of the function $u$, by Rademacher's theorem.

\begin{prop}[{Proposition~2.1 in~\cite{alambca}}]\label{supcontinu}
Let the function $u:M \to \rl$ be semiconcave, then the superdifferential 
$\supdiff u$ is not empty at each point, moreover, 
$\supdiff v$ is {\em upper semicontinuous}, namely
$$
q_k\lra q,\quad
  v_k\lra v,\quad  v_k\in \supdiff u(q_k)\quad \Longrightarrow \quad
  v\in \supdiff u(q).
$$
In particular, if the differential $du$ exists at {\em every} point
of $M$, then $u\in C^1(M)$.
\end{prop}

\begin{prop}[{Remark~3.6 in~\cite{alambca}}]\label{estremali} The set $Ext(\partial^+d_p(q)$ of
extremal points of the (convex) superdifferential set of 
$d_p$ at $q$ is in one--to--one correspondence with the family
${\mathcal G}(q)$ of minimal geodesics from $p$ to $q$. Precisely
${\mathcal G}(q)$ is described by
$$
{\mathcal G}(q)=\Bigl\{\,\exp_q(-vt),\,\, \text{$t\in[0,1]$}\,\,\Bigr|\,\,\text{$\forall v\in Ext(\partial^+d_p(q)$}\,\Bigr\}\,.
$$
\end{prop}

We now deal with structure of the cutlocus of $p\in M$. Let $\H^{n-1}$ denote the $(n-1)$--dimensional Hausdorff measure on $(M,g)$ (see~\cite{fede,simon}).

\begin{dfnz} We say that a subset $S\subset M$ is $C^r$--{\em
    rectifiable}, with $r\geq1$, if it can be covered by a countable family of
  embedded $C^r$--submanifolds of dimension $(n-1)$, with the
  exception of a set of $\H^{n-1}$--zero measure 
  (see~\cite{fede,simon} for a complete discussion of the notion of
  rectifiability).
\end{dfnz}

\begin{prop}[{Theorem~4.10 in~\cite{mant4}}] The cutlocus of $p\in M$ is $C^{\infty}$--rectifiable. Hence, its Hausdorff dimension is at most $n-1$. Moreover, for any compact subset $K$ of $M$ the measure $\H^{n-1}(\cutp\cap K)$ is finite (Corollary~1.3 in~\cite{nirenli}). 
\end{prop}

To explain the following consequence of such rectifiability, we need to introduce briefly the theory 
of functions with {\em bounded variation}, see~\cite{ambfupa,braides,fede,simon} for
details. We say that a function $u:\R^n\to\R^m$ is a function with
{\em locally bounded variation}, that is, $u\in BV_{\mathrm {loc}}$, if its distributional
derivative $Du$ is a Radon measure. Such notion can be easily extended 
to maps between manifolds using smooth local charts.\\
A standard result says that the derivative of a locally
semiconcave function stays in $BV_{\mathrm {loc}}$, in view of
Proposition~\ref{carlosemicref} this implies that the vector field
$\nabla d_p$ belongs to $BV_{\mathrm {loc}}$ in the open set $M\setminus\{p\}$.

Then, we define the subspace of $BV_{\mathrm {loc}}$ of functions (or vector
fields, as before) with locally {\em special bounded variation}, called
$SBV_{\mathrm {loc}}$ (see~\cite{amb1,amb3,amb4,ambfupa,braides}).\\
The Radon measure representing the distributional  derivative $Du$ of
a function $u:\R^n\to\R^m$ with locally bounded variation can be
always uniquely separated in three mutually singular measures
\begin{equation*}
Du=\widetilde{Du}+\mathrm{J}u+\mathrm{C}u
\end{equation*}
where the first term is the part absolutely continuous with respect to 
the Lebesgue measure ${\mathcal{L}}^n$, $\mathrm{J}u$ is a measure concentrated
on an $(n-1)$--rectifiable set and $\mathrm{C}u$ (called the {\em Cantor part}) 
is a measure which does not charge the subsets of Hausdorff dimension
$(n-1)$.\\
The space $SBV_{\mathrm {loc}}$ is defined as the class of functions $u\in
BV_{\mathrm {loc}}$ such that $\mathrm{C}u=0$, that is, the Cantor part of the
distributional derivative of $u$ is zero. Again, by means of the local charts, this notion is easily 
generalized to Riemannian manifolds.

\begin{prop}[{Corollary~4.13 in~\cite{mant4}}] The (${\mathcal{H}}^{n-1}$--almost everywhere defined) measurable unit vector field
  $\nabla d_p$ belongs to the space $SBV_{\mathrm {loc}}(M\setminus\{p\})$ of
  vector fields with locally special bounded variation.
\end{prop}

The immediate consequence of this proposition is that the $(0,2)$--tensor field valued distribution ${\mathrm{Hess}}\,d_p$ is actually a Radon measure with an absolutely continuous part, with respect to the canonical volume measure ${\mathrm{Vol}}$ of $(M,g)$, concentrated in 
$M\setminus(\{p\}\cup \cutp)$ where $d_p$ is a smooth function, hence in this set ${\mathrm{Hess}}\,d_p$ coincides with the standard Hessian ${\widetilde{\mathrm{Hess}}}\, d_p$ times the volume measure ${\mathrm{Vol}}$. 
When the dimension of $M$ is at least two, the singular part of the measure ${\mathrm{Hess}}\,d_p$ does not ``see'' the singular point $p$, hence, it is concentrated on $\cutp$, absolutely continuous with respect to the Hausdorff measure ${\mathcal{H}}^{n-1}$, restricted to $\cutp$.

By the properties of rectifiable sets, at ${\mathcal{H}}^{n-1}$--almost every point $q\in\cutp$ there exists an $(n-1)$--dimensional {\em approximate tangent space} $\apT_q\cutp\subset T_qM$ (in the sense of geometric measure theory, see~\cite{fede,simon} for details). To give an example, we say that an hyperplane $T\subset\R^n$ is the approximate tangent space to an $(n-1)$--dimensional rectifiable set $K\in\R^n$ at the point $x_0$, if ${\mathcal{H}}^{n-1}\res T$ is the limit in the sense of Radon measures, as $\rho\to+\infty$, of the blow--up measures ${\mathcal{H}}^{n-1}\res\rho(K-x_0)$ around the point $x_0$. With some technicalities, this notion can be extended also to Riemannian manifolds.\\
Moreover, see~\cite{ambfupa}, at ${\mathcal{H}}^{n-1}$--almost every point $q\in\cutp$, the field $\nabla d_p$ has two distinct {\em approximate} (in the sense of Lebesgue differentiation theorem) limits ``on the two sides'' of $\apT_q\cutp\subset T_qM$, given by $\nabla d^+_p$ and $\nabla d^-_p$.

We want to see now that at ${\mathcal{H}}^{n-1}$--almost
every point of $\cutp$ there arrive exactly two distinct geodesics and
no more. We underline that a stronger form of this theorem was already
obtained in~\cite{ardgui} and~\cite{figrifvil}, concluding that the
set $\cutp\setminus U$ (where $U$ is like in the following statement) actually has Hausdorff dimension not greater that $n-2$.

\begin{teo}\label{smoo}
There is an open set $U\subset M$ such that ${\mathcal{H}}^{n-1}(\cutp\setminus U)=0$ and 
\begin{itemize}
\item the subset $\cutp\cap U$ does not contain conjugate points, hence the set of optimal conjugate points has ${\mathcal{H}}^{n-1}$--zero measure;
\item at every point of $\cutp\cap U$ there arrive exactly two minimal geodesics from $p\in M$;
\item locally around every point of $\cutp\cap U$ the set $\cutp$ is a smooth $(n-1)$--dimensional 
hypersurface, hence, $\apT_q\cutp$ is actually the classical tangent space to a hypersurface.
\end{itemize}
\end{teo}
\begin{proof}
First we show that the set of optimal conjugate points $\Conj$ is a closed subset of ${\mathcal{H}}^{n-1}$--zero measure, then we will see that the points of $\Sing\setminus\Conj$ where there arrive more than two geodesics is also a closed subset of ${\mathcal{H}}^{n-1}$--zero measure. The third point then follows by the analysis in the proof of Proposition~4.7 in~\cite{mant4}.

Recalling that $U_p=\left\{v\in T_pM\,\vert \, g_p(v,v)=1 \right\}$ is the
set of unit tangent vectors to $M$ at $p$, we define the function
$c:U_p\to\R^+\cup\{+\infty\}$ such that the point $\gamma_v(c_v)$ is
the first conjugate point (if it exists) along the geodesic
$\gamma_v$, that is, the differential $d\exp_p$ is not invertible at
the point $c_vv\in T_pM$. By Lemma~4.11 and the proof of
Proposition~4.9 in~\cite{mant4}, in the open subset $V\subset U_p$
where the rank of the differential of the map $F:U_p\to M$, defined as
$F(v)=\exp_p(c_vv)$ is $n-1$, the map $c:U_p\to\R^+\cup\{+\infty\}$ is
smooth hence $F(V)$ is locally a smooth hypersurface. As, by Sard's
theorem, the image of $U_p\setminus V$ is a closed set of
${\mathcal{H}}^{n-1}$--zero measure, we only have to deal with the
images $F(v)$ of the unit vectors $v\in V$ with $c_v=\sigma_v$ (see
at the end of the introduction), that is, with $F(V)\cap\cutp$, which is a
closed set.\\
We then consider the set $D\subset (F(V)\cap\cutp)$ of the points $q$
where $\apT_q\cutp$ exists and the {\em density} of the rectifiable
set $F(V)\cap\cutp$ in the cutlocus of the point $p$, with respect to
the Hausdorff measure ${\mathcal{H}}^{n-1}$, is one
(see~\cite{fede,simon}). It is well known that $D$ and $F(V)\cap\cutp$
only differ by a set of ${\mathcal{H}}^{n-1}$--zero measure. If
$F(v)=q\in D$ then, $c_v=\sigma_v$ and, by the above density property,
the hypersurface $F(V)$ is ``tangent'' 
to $\cutp$ at the point $q$, that is, $T_qF(V)=\apT_q\cutp$.\\
We claim now that the minimal geodesic $\gamma_v$ is tangent to the
hypersurface $F(V)$, hence to the cutlocus, at the point $q$. Indeed,
as $d\exp_p$ is not invertible at $c_vv\in T_pM$, by Gauss lemma there
exists a vector $w\in T_vU_p$ such that $d\exp_p[c_vv](w)=0$, hence
$$
dF_v(w)=(dc[v](w))\dot{\gamma}_v(c_v)+d\exp_p[c_vv](c_vw)
=(dc[v](w))\dot{\gamma}_v(c_v)\,,
$$
thus, $\dot{\gamma}_v(c_v)$ belongs to the tangent space $dF(T_vU_p)$
to the hypersurface $F(V)$ at the point $q$, which coincides with
$\apT_q\cutp$, as we claimed.\\
By the properties of SBV functions described before, at ${\mathcal{H}}^{n-1}$--almost every point $q\in D$, the blow--up of the function $d_p$ is a ``roof'', that is, there arrive exactly two minimal geodesics both intersecting transversally the cutlocus at $q$ (the vectors $\nabla d^+_p$ and $\nabla d^-_p$ do not belong to $\apT_qM$), hence the above minimal geodesic $\gamma_v$ cannot coincide with any of these two.\\ 
We then conclude that ${\mathcal{H}}^{n-1}(D)=0$ and the same for the set $\Conj$.

Suppose now that $q\in\cutp\setminus\Conj\subset\Sing$, by the
analysis in the proof of Proposition~4.7 in~\cite{mant4} (and
Lemma~4.8), at the point $q$ there arrive a {\em finite} number $m\geq
2$ of distinct minimal geodesics and when $m>2$ the cutlocus of $p$ is
given by the union of at least $m$ smooth hypersurfaces with Lipschitz
boundary passing at the point $q$, in particular the above blow--up at
$q$ cannot be a single hyperplane $\apT_q\cutp$. By the above
discussion, such points with $m>2$ are then of
${\mathcal{H}}^{n-1}$--zero measure, moreover, by
Propositions~\ref{supcontinu} and~\ref{estremali} the set of points in
$\cutp\setminus\Conj$ with only two minimal geodesics is open and we
are done.
\end{proof}

\begin{rem} In the special two--dimensional and analytic case, it can
  be said something more, that is, the number of optimal conjugate
  points is locally finite and the cutlocus is a locally finite graph
  with smooth edges, see the classical papers by
  Myers~\cite{myers1,myers2}. We conjecture
  that, in general, the set of optimal conjugate points is 
  an $(n-2)$--dimensional rectifiable set. 
\end{rem}

By the third point if this theorem, in the open set $U$ the two side limits $\nabla d^+_p$ and $\nabla d^-_p$ of the gradient field $\nabla d_p$ are actually smooth and classical limits, moreover it is locally defined a smooth unit normal vector $\nu_q\in T_qM$ orthogonal to $T_q\cutp$, with the convention that $g_q(\nu_q,v)$ is positive for every vector $v\in T_qM$ belonging to the halfspace corresponding to the side associated to $\nabla d^+_p$. Hence, since ${\mathcal{H}}^{n-1}(\cutp\setminus U)=0$, we have a precise description of the singular ``jump'' part as follows,
$$
{\mathrm{J}}\nabla d_p=-\Bigl(\,(\nabla d^+_p-\nabla d^-_p)\otimes\nu\,\Bigr)\,{\mathcal{H}}^{n-1}\res \cutp
$$
and, noticing that the ``jump'' of the gradient of $d_p$ in $U$ must be orthogonal to the tangent space $T_q\cutp$, thus parallel to the unit normal vector $\nu_q\in T_qM$, hence we conclude
$$
{\mathrm{J}}\nabla d_p=-(\nu\otimes\nu)\,\vert\nabla d^+_p-\nabla d^-_p\vert_g\,{\mathcal{H}}^{n-1}\res \cutp\,.
$$
Notice that the singular part of the distributional Hessian of $d_p$ is a rank one symmetric $(0,2)$--tensor field.

\begin{rem}\label{extra} This description of the ``jump'' part of the singular measure is actually a direct consequence of the structure theorem of BV functions (see~\cite{ambfupa}), even without knowing, by Theorem~\ref{smoo}, that the cutlocus is ${\mathcal{H}}^{n-1}$--almost everywhere smooth.
\end{rem}

\begin{teo}
If $n\geq 2$, the distributional Hessian of the distance from a point $p\in M$ is given by the Radon measure 
$$
{\mathrm{Hess}}\,d_p={\widetilde{\mathrm{Hess}}}\,d_p\,{\mathrm{Vol}}-(\nu\otimes\nu)\,\vert\nabla d^+_p-\nabla d^-_p\vert_g\,{\mathcal{H}}^{n-1}\res \cutp\,,
$$
where ${\widetilde{\mathrm{Hess}}}\,d_p$ is the standard Hessian of $d_p$, where it exists (${\mathcal{H}}^{n-1}$--almost everywhere on $M$), and $\nabla d^+_p$, $\nabla d^-_p$, $\nu$ are defined above.
\end{teo}

\begin{cor} If $n\geq 2$, the distributional Laplacian of $d_p$ is the Radon measure
$$
\Delta d_p=\widetilde{\Delta}d_p\,{\mathrm{Vol}}-\vert\nabla d^+_p-\nabla d^-_p\vert_g\,{\mathcal{H}}^{n-1}\res \cutp\,,
$$
where ${\widetilde{\Delta}} d_p$ is the standard Laplacian of $d_p$, where it exists.
\end{cor}

\begin{cor} There hold
$$
\Delta d_p\leq\widetilde{\Delta}d_p\,{\mathrm{Vol}}
$$
and
$$
{\mathrm{Hess}}\,d_p\leq {\widetilde{\mathrm{Hess}}}\,d_p\,{\mathrm{Vol}}\,,
$$
as $(0,2)$--tensor fields.\\
As a consequence, the Hessian and Laplacian inequalities in Theorem~\ref{ggg} hold in the sense of distributions.\\
Moreover, we have
$$
\Delta d_p\geq\widetilde{\Delta}d_p\,{\mathrm{Vol}}-2{\mathcal{H}}^{n-1}\res \cutp\,.
$$
and	
$$
{\mathrm{Hess}}\,d_p\geq {\widetilde{\mathrm{Hess}}}\,d_p\,{\mathrm{Vol}}-2(\nu\otimes\nu)\,{\mathcal{H}}^{n-1}\res \cutp\geq {\widetilde{\mathrm{Hess}}}\,d_p\,{\mathrm{Vol}}-2g\,{\mathcal{H}}^{n-1}\res \cutp\,,
$$
as $(0,2)$--tensor fields.
\end{cor}

\begin{rem} By their definition, it is easy to see that the same
  inequalities hold also for the Busemann functions, see for
  instance~\cite[Subsection~9.3.4]{petersen1} (in Section~9.3 of the same book 
  it is shown that the above Laplacian comparison holds on the
  whole $M$ in barrier sense while the analogous result for the Hessian
  can be found in Section~11.2). 
We underline here that Propositions~\ref{carlosemicref},~\ref{supcontinu} and~\ref{estremali} about the semiconcavity and the structure of the superdifferential of the distance function $d_p$ can also be used to show that the above inequalities hold in barrier/viscosity sense.
\end{rem}

\begin{rem}
Several of the conclusions of this paper holds also for the distance function from a closed {\em subset} of $M$ with boundary of class $C^3$ at least, see~\cite{mant4} for details.
\end{rem}

\appendix

\section{Weak Definitions of Sub/Supersolutions of PDEs}
\label{appA}

Let $(M,g)$ be a smooth, complete, Riemannian manifold and let $A$ be a smooth $(0,2)$ tensor field.

If $f:M\to\R$ satisfies ${\mathrm{Hess}}\,f\leq A$ at the point $p\in
M$ in barrier sense, for every $\varepsilon>0$ there exists a
neighborhood $U_\varepsilon$ of the point $p$ and a $C^2$--function
$h_\varepsilon:U_\varepsilon\to\R$ such that $h_\varepsilon(p)=f(p)$,
$h_\varepsilon\geq f$ in $U_\varepsilon$ and ${\mathrm{Hess}}\,
h_\varepsilon(p)\leq A(p)+\varepsilon g(p)$, hence, 
every $C^2$--function $h$ from a neighborhood $U$ of the point $p$
such that $h(p)=f(p)$ and $h\leq f$ in $U$ satisfies
$h(p)=h_\varepsilon(p)$ and $h\leq h_\varepsilon$ in $U\cap
U_\varepsilon$. It is then easy to see that it must be
${\mathrm{Hess}}\, h(p)\leq {\mathrm{Hess}}\, h_\varepsilon(p)\leq
A(p)+\varepsilon g(p)$, for every $\varepsilon>0$, hence 
${\mathrm{Hess}}\, h(p)\leq A(p)$. This shows that ${\mathrm{Hess}}\,
f\leq A$ at the point $p\in M$ also in {\em viscosity sense}.\\
The converse is not true, indeed, it is straightforward to check that
the function $f:\R\to\R$ given by $f(x)=x^2\sin{(1/x)}$ when $x\not=0$
and $f(0)=0$ satisfies $f^{\prime\prime}(0)\leq 0$ in viscosity sense but
not in barrier sense.\\
The same argument clearly also applies to the two definitions of $\Delta
f\leq\alpha$, for a smooth function $\alpha:M\to\R$.

We see now that instead the definitions of viscosity and
distributional sense coincide.

\begin{prop}\label{prop-equivalence} If $f:M\to\R$ satisfies ${\mathrm{Hess}}\, f\leq A$ in
  viscosity sense, it also satisfies  ${\mathrm{Hess}}\, f\leq A$ in 
distributional sense and viceversa. The same holds for $\Delta f\leq\alpha$.
\end{prop}

In order to show the proposition, we recall the definitions of {\em
  viscosity} (sub/super) solution to a second order PDE.
Take a continuous map
$F:\Omega\times\R\times\R^n\times S^{n}\to\R$, where $\Omega$ is an open subset of $\R^n$ and $S^n$
denotes the space of real $n\times n$ symmetric matrices; also suppose that $F$ satisfies 
the {\em
  monotonicity condition}
$$
X\geq Y\qquad\Longrightarrow\qquad F(x,r,p,X)\leq F(x,r,p,Y)\,,
$$
for every $(x,r,p)\in\Omega\times\R\times\R^n$, where $X\geq Y$ means
that the difference matrix $X-Y$ is nonnegative definite. We consider then the second order PDE given by $F(x,f,\nabla f,\nabla^2f)=0$.

A continuous function $f:\Omega\to\R$ is said a {\em viscosity
  subsolution} of the above PDE if for every point $x\in\Omega$ and $\varphi\in C^2(\Omega)$
such that $f(x)-\varphi(x)= \sup_\Omega (f-\varphi)$, 
there holds $F(x,\varphi,\nabla\varphi,\nabla^2\varphi)\leq 0$

(see~\cite{crisli1,ishii2}). 
Analogously, $f\in C^0(\Omega)$ is a {\em viscosity
  supersolution} if for every point $x\in\Omega$ and $\varphi\in C^2(\Omega)$
such that $f(x)-\varphi(x)= \inf_\Omega (f-\varphi)$, 
there holds $F(x,\varphi,\nabla\varphi,\nabla^2\varphi)\geq 0$. If
$f\in C^0(\Omega)$ is both a viscosity subsolution and supersolution, 
it is then a {\em viscosity solution} of $F(x,f,\nabla
f,\nabla^2f)=0$ in $\Omega$.

It is easy to see that the functions $f\in C^0(\Omega)$ such that $\Delta
f\leq\alpha$ in {\em viscosity sense} at any point of $\Omega$, as in Definition~\ref{weakdef}, 
coincide with the viscosity supersolutions of the equation $-\Delta f+\alpha=0$ at the same point
(here the function $F$ is given by $F(x,r,p,X)=-\trace{X}+\alpha(x)$).

In the case of a Riemannian manifold $(M,g)$, one works in local charts and the operators
we are interested in become
\begin{align*}
\Hess^M_{ij} f(x) & = \frac{\partial^2 f(x)}{\partial x^i \partial x^j} - \Gamma_{ij}^k(x) \frac{\partial f}{\partial x^k} \\
\Delta^M f(x) & = g^{ij}(x) \Hess^M_{ij} f(x)\,,
\end{align*}
where $\Gamma_{ij}^k$ are the Christoffel symbols.\\
Analogously to the case of $\R^n$, taking 
$F(x,r,p,X)=-g^{ij}(x)X_{ij}+g^{ij}(x)\Gamma_{ij}^k(x)p_k+\alpha(x)$
(which is a smooth function independent of the variable $r$), we see 
that, according to Definition~\ref{weakdef}, $f$ satisfies $\Delta^M f\leq\alpha$ in {\em viscosity sense} at any point of $M$ 
if and only if it is a viscosity supersolution of the equation $F(x,f,\nabla f,\nabla^2f)=0$ at the same point.

Getting back to $\R^n$, given a linear, degenerate elliptic operator $L$ with smooth coefficients,
that is, defined by
$$
Lf(x)=-a^{ij}(x)\nabla^2_{ij}f(x)+b^k(x)\nabla_kf(x)+c(x)f(x)\,,
$$
and a smooth function $\alpha:\Omega\to\R$, 
we say that $f\in C^0(\Omega)$ is a distributional supersolution of the equation $Lf+\alpha=0$ when
$$
\int_\Omega \Bigl( fL^*\varphi+\alpha\varphi\, \Bigr)dx\geq0
$$
for every nonnegative, smooth function $\varphi\in C_c^\infty(\Omega)$.
Here $L^*$ is the formal adjoint operator of $L$:
\[
L^* \varphi(x) = -\nabla^2_{ji}(a^{ij}\varphi)(x) - \nabla_k(b^k\varphi)(x) + c(x)\varphi(x) .
\]

Under the hypothesis that the matrix of coefficients $(a_{ij})$ (which
is nonnegative definite) has a ``square root'' matrix belonging to $C^1(\Omega,S^n)$, Ishii showed in paper~\cite{ishii2} the equivalence of
the class of continuous viscosity subsolutions and the class
of continuous distributional subsolutions of the equation $Lf+\alpha=0$. 
More precisely, he proved the following two theorems (see also~\cite{lions2}).
\begin{teo}[Theorem~1 in~\cite{ishii2}]\label{Iteo1}
If $f\in C^0(\Omega)$ is a viscosity subsolution of the equation $Lf+\alpha=0$, then 
then it is a distribution subsolution of the same equation.
\end{teo}
\begin{teo}[Theorem~2 in~\cite{ishii2}]\label{Iteo2}
Assume that the ``square root'' of the matrix of coefficients $(a_{ij})$ belongs to $C^1(\Omega)$. 
If $f\in C^0(\Omega)$ is a distributional subsolution of the equation $Lf+\alpha=0$, then 
then it is a viscosity subsolution of the same equation.
\end{teo}
As the PDE is linear, a function $f\in C^0(\Omega)$ is a viscosity (distributional) supersolution of the equation $Lf+\alpha=0$ if and only if 
the function $-f$ is a viscosity (distributional) subsolution of $L(-f)-\alpha=0$, in the above theorems every occurrence of the 
term ``subsolution'' can replaced with ``supersolution'' (and actually
also with ``solution'').

For simplicity, we will work in a single coordinate chart of $M$ mapping onto $\Omega\subseteq
\R^n$, while the general situation can be dealt with by means of standard partition of unity arguments.\\
Consider $f \in C^0(M)$ which is a viscosity supersolution of $-\Delta^M f + \alpha = 0$.
It is a straightforward computation to check that this happens
if and only if $f$ is a viscosity supersolution of $-\sqrt{g}\Delta^M f +
\alpha\sqrt{g} = 0$, where $\sqrt{g} = \sqrt{\det g_{ij}}$ is the density of
Riemannian volume of $(M,g)$, and viceversa. Moreover, notice that
setting $L=-\sqrt{g}\Delta^M$ we have that $L^*=L$, that is, $L$ is a
self--adjoint operator; it also satisfies the hypotheses of
Ishii's theorems, being the matrices $g_{ij}$ and $g^{ij}$ smooth and
positive definite in $\Omega$ (see~\cite[Chapter~6]{Horn}, in particular
Example~6.2.14, for instance).

Then, in local coordinates, Ishii's theorems
guarantee that $f$ is a distributional supersolution of the same
equation, that is, $f$ satisfies, for each $\varphi \in
C^\infty_c(\Omega)$, 
$$
\int_\Omega f L^* \varphi \, dx \geq - \int_\Omega \alpha\sqrt{g} \varphi \,
dx\,,
$$
hence,
$$
\int_M-f\Delta^M\varphi\,d{\mathrm{Vol}}=\int_\Omega -f
\sqrt{g}\Delta^M \varphi \, dx 
\geq - \int_\Omega \alpha\sqrt{g} \varphi \,dx= -\int_M\alpha\varphi\,d{\mathrm{Vol}}\,.
$$
This shows that then $f$ satisfies $\Delta^M f \leq \alpha$ in distributional sense, as in
Definition~\ref{weakdef}.

Following these steps in reverse order, one gets the converse. Hence, the notions of
$\Delta^M\leq \alpha$ in viscosity and distributional sense coincide.

Now we turn our attention to the Hessian inequality; it is not covered 
directly by Ishii's theorems, which are peculiar to PDEs and do not deal 
with {\em systems} (like the general theory of viscosity
solutions). For simplicity, we discuss the case of an open set $\Omega\subset\R^n$ 
(with its canonical flat metric), since all the arguments can be extended to
any Riemannian manifold $(M,g)$ by localization and introduction of the first--order correction given by Christoffel symbols, as above.

The idea is to transform the matrix inequality $\Hess\,f\leq A$ 
into a family of scalar inequalities; indeed, if everything is
smooth, such inequality is satisfied if and only if for every compactly--supported, smooth vector
field $W$ we have $W^iW^j\Hess_{ij}f\leq A_{ij}W^iW^j$. The only price
to pay is that we lose the constant coefficients of the Hessian, hence
making the linear operator $L^W$, acting on $f\in C^2(\Omega)$ as $L^Wf=-W^iW^j\Hess_{ij}f$,
only {\em degenerate} elliptic. Notice that Ishii's condition in Theorem~\ref{Iteo2} is
satisfied for every smooth vector field $W$ such that $\Vert W\Vert\in C^1_c(\Omega)$,
but not by any arbitrary smooth vector field. This has the collateral effect of making
the proof of the Hessian case in Proposition~\ref{prop-equivalence} slightly asymmetric.

\begin{lemma}\label{lemmakey} Let $f\in C^0(\Omega)$. If for every compactly--supported, smooth vector
field $W$ with $\Vert W\Vert\in C^1_c(\Omega)$, 
we have that $f$ is a viscosity supersolution of the equation 
$-W^iW^j\Hess_{ij}f+A_{ij}W^iW^j=0$, then the function $f$ satisfies $\Hess\,f\leq A$ in viscosity sense in the whole $\Omega$.

Viceversa, if $f\in C^0(\Omega)$ satisfies $\Hess\,f\leq A$ in viscosity sense in $\Omega$, then 
$f$ is a viscosity supersolution of the equation 
$-V^iV^j\Hess_{ij}f+A_{ij}V^iV^j=0$ for every compactly--supported, smooth vector
field $V$.
\end{lemma}
\begin{proof}
Let us take a point $x \in \Omega$ and a $C^2$--function $h$ in 
a neighborhood $U$ of the point $x$ such that $h(x) = f(x)$ and $h \leq f$. Choosing a unit vector
$W_x$ and a smooth, nonnegative function $\varphi$, which is 1 at $x$ and zero outside a small ball inside $U$, we consider the smooth vector field 
$W(y)=W_x\varphi^2(y)$, for every $y\in\Omega$, 
which clearly satisfies $\Vert W\Vert=\varphi\in C^1_c(\Omega)$. By the hypothesis of the first statement, 
the function $f$ is then a viscosity supersolution of the equation 
$-W^iW^j\Hess_{ij}f+A_{ij}W^iW^j=0$ which implies that
$-W_x^iW_x^j\Hess_{ij}h(x)+A_{ij}(x)W_x^iW_x^j\geq0$.
Since this holds for every point $x\in\Omega$ and unit vector $W_x$, we conclude that 
$\Hess\,h(x)\leq A(x)$ as $(0,2)$--tensor fields, hence $\Hess\,f\leq A$ in viscosity sense in $\Omega$.

The argument to show the second statement is analogous: given a compactly--supported, smooth vector
field $V$, a point $x\in\Omega$ and a function $h$ as above, the hypothesis implies that 
$-V_x^iV_x^j\Hess_{ij}h(x)+A_{ij}(x)V_x^iV_x^j\geq0$, hence the thesis.
\end{proof}

Suppose now that $f\in C^0(\Omega)$ satisfies $\Hess\,f\leq A$ in viscosity sense on the whole $\Omega$; hence, by this lemma, 
for every compactly--supported, smooth vector field $V$, the function 
$f$ is a viscosity supersolution of the equation $-V^iV^j\Hess_{ij}f+A_{ij}V^iV^j=0$. By Theorem~\ref{Iteo1} and the subsequent discussion, it is then a 
distributional supersolution of the same equation, that is,
$$
\int_\Omega \Bigl[-f\nabla^2_{ji}(V^iV^j\varphi)+A_{ij}V^iV^j\varphi\Bigr]\,dx\geq0
$$
for every nonnegative, smooth function $\varphi\in C_c^\infty(\Omega)$.\\
Considering a nonnegative, smooth function $\varphi\in C_c^\infty(\Omega)$ such that it is one on the support of the vector field $V$ we conclude
$$
\int_\Omega f\nabla^2_{ji}(V^iV^j)\,dx\leq\int_\Omega A_{ij}V^iV^j\,dx\,,
$$
which means that $\Hess f\leq A$ in distributional sense.

Conversely, if $f\in C^0(\Omega)$ satisfies $\Hess\,f\leq A$ in distributional sense, then for every compactly--supported, smooth vector
field $W$ with $\Vert W\Vert\in C^1_c(\Omega)$ and smooth, nonnegative function $\varphi\in C^\infty_c(\Omega)$, 
we define the smooth, nonnegative functions $\varphi_n = \varphi + \psi/n$, where $\psi$ is a smooth, nonnegative and
compactly--supported function such that $\psi \equiv 1$ on the support
of $W$. It follows that the vector field $V={W}{\sqrt{\varphi_n}}$ is smooth, hence, applying the definition of 
$\Hess\,f\leq A$ in distributional sense, we get
\begin{equation*}
\int_\Omega \Bigl[-f\nabla^2_{ji}(W^iW^j\varphi_n)+A_{ij}W^iW^j\varphi_n\Bigr]\,dx\geq0\,.
\end{equation*}
As $\varphi_n \to \varphi$ in $C^\infty_c(\Omega)$ and $f$ is continuous, we can pass to the limit in $n\to\infty$ and conclude that
\begin{equation*}
\int_\Omega \Bigl[-f\nabla^2_{ji}(W^iW^j\varphi)+A_{ij}W^iW^j\varphi\Bigr]\,dx\geq0\,,
\end{equation*}
for every nonnegative, smooth function $\varphi\in C_c^\infty(\Omega)$ and 
every compactly--supported, smooth vector
field $W$ with $\Vert W\Vert\in C^1_c(\Omega)$. That is, for any vector field $W$ as above, we have that $f$ is a 
distributional supersolution of the equation $-W^iW^j\Hess_{ij}f+A_{ij}W^iW^j=0$.\\
By Theorem~\ref{Iteo2} and the subsequent discussion, it is then a 
viscosity supersolution of the same equation and, by Lemma~\ref{lemmakey}, we conclude that the function $f$ satisfies $\Hess\,f\leq A$ in viscosity sense.

Summarizing, we have the following sharp relations among the weak
notions of the partial differential inequalities
$\Hess\,f\leq A$ and $\Delta f\leq\alpha$,
$$
\text { barrier sense }\qquad\Longrightarrow\qquad\text{ viscosity sense }\qquad\Longleftrightarrow\qquad\text { distributional sense. }
$$

\bibliographystyle{amsplain}
\bibliography{biblio}

\end{document}